\documentclass[reqno]{amsart}

\usepackage[T1]{fontenc}
\usepackage[utf8]{inputenc}

\usepackage{amssymb,amsmath,amsthm,amscd,latexsym,amsfonts,xcolor,enumerate,hyperref}
\usepackage{graphicx}
\usepackage{mathtools}
\usepackage{orcidlink}
\newtheorem{theorem}{Theorem}[section]
\newtheorem{lemma}[theorem]{Lemma}

\theoremstyle{definition}
\newtheorem{definition}[theorem]{Definition}
\newtheorem{example}[theorem]{Example}

\theoremstyle{remark}
\newtheorem{remark}[theorem]{Remark}

\numberwithin{equation}{section}

\DeclareMathOperator{\ad}{ad}

\DeclareMathOperator{\chr}{char}

\DeclareMathOperator{\Der}{Der}
\DeclareMathOperator{\End}{End}

\DeclareMathOperator{\Id}{Id}

\begin{document}

\title{HNN-extension of Lie superalgebras}

\author{M.~Ladra \orcidlink{0000-0002-0543-4508}}
\address{[M.~Ladra] Departamento de Matem\'aticas, Instituto de Matem\'aticas, Universidade de Santiago de Compostela,
Santiago de Compostela, Spain}
\email{manuel.ladra@usc.es}

\author{P.~P\'aez-Guill\'an \orcidlink{0000-0003-2761-7505}}
\address{[P.~P\'aez-Guill\'an] Departamento de Matem\'aticas, Instituto de Matem\'aticas, Universidade de Santiago de Compostela,
Santiago de Compostela, Spain}
\email{pilar.paez@usc.es}

\author{C.~Zargeh \orcidlink{0000-0002-1763-1568}}
\address{[C.~Zargeh] Departamento de Matem\'aticas, Instituto de Matem\'aticas, Universidade de Santiago de Compostela,
Santiago de Compostela, Spain}
\email{chia.zargeh@usc.es}

\date{2018}

\begin{abstract}
We construct HNN-extensions of Lie superalgebras and prove that every  Lie superalgebra embeds into
any of its HNN-extensions. Then as an application we show that any Lie superalgebra with at
most countable dimension embeds into a two-generator Lie superalgebra.
\end{abstract}


\maketitle

\section*{Introduction}
The Higman-Neumann-Neumann extensions (HNN-extensions) for groups  were introduced in \cite{H1} in the context of the study of embeddability of groups. The idea of the construction of HNN-extensions has been extended in several
ways to other algebraic structures such as Lie algebras \cite{L3,W1}, semigroups \cite{H2},  associative rings \cite{D1,L3}  and Leibniz algebras \cite{L1}.
The HNN-extensions are a strong tool in combinatorial group theory to introduce embedding theorems, as they allow to prove
that every countable group can be embedded into a group with two generators. There exist embedding theorems for the case of Lie algebras and Leibniz algebras
 analogous to the ones in group theory stating that every Lie (Leibniz) algebra with at most countable dimension can be embedded into a Lie (Leibniz) algebra with two generators.
   As the main difference between the construction of HNN-extensions for groups and algebras, we note that HNN-extensions of groups are defined by subgroups and isomorphisms,
    whereas HNN-extensions of algebras come from subalgebras and derivations.
     In addition to that, the approach to the construction of HNN-extensions of Lie (Leibniz) algebras is different
     and it is based on the theory of Gr\"obner-Shirshov bases and Composition-Diamond lemma (see \cite{B1}).

In this paper, we use the Gr\"obner-Shirshov bases and the Composition-Diamond lemma for Lie superalgebras, introduced in \cite{B2}, to construct the HNN-extensions of Lie superalgebras, and to prove that every Lie superalgebra can be embedded into its HNN-extension. The embeddability of countable dimensional Lie superalgebras into two generated Lie superalgebras was proved by Mikhalev in \cite{M1}.

The paper is organized as follows. In Section~\ref{S:prel} we
recall some necessary notions and facts on Lie superalgebras.
 In Section~\ref{S:GS},  following the paper on Lie superalgebras \cite{B2},  we remember  the construction of Gr\"obner–Shirshov bases for Lie superalgebras and their properties.
In  Section~\ref{S:HNN},  we  construct the HNN-extension of Lie superalgebras and prove that every  Lie superalgebra embeds into
any of its HNN-extensions.
Finally, in  Section~\ref{S:app}  we give a different proof of the embeddability  theorem of countable dimensional Lie superalgebras into two generated Lie superalgebras, using the HNN-extension.

%%%%%%%%%%%%%%%%%%%%%%%%%%%%%%%%%%%%%%%%%%%%%%%%%%%%%%%%%%%%%%%%%%%%%%%%
\section{Preliminaries}\label{S:prel}
In this section we recall notions and properties of Lie superalgebras.

If $V$ is a $\mathbb{Z}_2$-graded vector space over a field $\mathbb{K}$, then $V=V_{\bar{0}} \oplus V_{\bar{1}}$, where the even and the odd subspaces of $V$ are denoted by $V_{\bar{0}}$ and   $V_{\bar{1}}$, respectively. The non-zero elements of $V_{\bar{0}} \cup V_{\bar{1}}$ will be called \emph{homogeneous}. Let us consider the gradation $|\ | \colon V \to \mathbb{Z}_2=\{\bar{0},\bar{1}\}$ for the homogenous elements of $V$ defined by
\[
  |x|=\begin{cases}
    \bar{0}, & \text{if} \ x \in V_{\bar{0}},\\
    \bar{1}, & \text{if} \  x \in V_{\bar{1}}.
  \end{cases}
\]
\begin{definition}
A \emph{Lie superalgebra} $L$ is a $\mathbb{Z}_2$-graded vector space $L=L_{\bar{0}}\oplus L_{\bar{1}}$ equipped with a bilinear superbracket structure $[~,~] \colon L \otimes L \to L$ preserving the degree (being $|a\otimes b|=|a|+|b|$) and satisfying the requirements of graded anti-symmetry and super Jacobi identity as follows:
\[[x,y]=-(-1)^{{|x|}{|y|}} [y,x], \]
\[[x,[y,z]]=[[x,y],z]+ (-1)^{{|x|}{|y|}} [y,[x,z]] ,\]
\[ [x_{\bar{0}}, x_{\bar{0}} ]=0,\]
\[[x_{\bar{1}},[x_{\bar{1}}, x_{\bar{1}} ]]=0,\]
for all homogeneous elements $x,y,z \in L$, $x_{\bar{0}} \in L_{\bar{0}}$ and $x_{\bar{1}} \in L_{\bar{1}}$. Note that the two last identities are trivial consequences of the first ones in case that $\chr(\mathbb{K})\neq 2,3$.
\end{definition}
 Let $\{X_i\}_{i\in\Lambda} \subset L_{\bar{0}}\cup L_{\bar{1}} $ be a basis of $L$, with structure constants given by
\[ [X_i,X_j]=\sum_{l\in\Lambda} \alpha_{ij}^{l} X_l .\]
Denoting $|X_i|$ by $|i|$, it is clear that $\alpha_{ij}^{k}=0$ whenever $|i|+|j| \neq |l|$. Structure constants satisfy the super Jacobi identity:
\begin{equation}\label{jacobi}
\sum_{l\in\Lambda}(\alpha_{jk}^l\alpha_{il}^{m}-\alpha_{ij}^l\alpha_{lk}^{m}-(-1)^{|i||j|}\alpha_{ik}^l\alpha_{jl}^{m})=0,
\end{equation}
where
\[ \alpha_{ij}^{l}= - (-1) ^{|i| |j|} \alpha_{ji}^{l}.\]
Also, they satisfy
\[\alpha_{ii}^l=0,\]
whenever $|i|=\bar{0}$, and
\[\sum_{l\in\Lambda}\alpha_{ii}^k\alpha_{il}^m=0,\]
if $|i|=\bar{1}$.
\begin{definition}
A derivation $d$ of degree $|d|$, $|d| \in \mathbb{Z}_2$, is defined as a linear map $d \colon L \to L$ such that
\[ d([a,b])=[d(a),b]+(-1)^{|d| |a|} [a,d(b)] .\]
We denote the space of derivations of degree $i$ as $\Der_i(L)$, and define the space of derivations as \[\Der(L)=\Der_{\bar{0}}(L)\oplus\Der_{\bar{1}}(L).\]
It holds that $\Der(L)$ is a graded Lie subalgebra of $\End(L)$.
\end{definition}
\begin{example}
Let $L$ be a Lie superalgebra, and $a\in L$. It follows from Jacobi identity that $\ad_a \colon L \to L, b \mapsto [a,b]$, is a derivation of $L$ of degree $|a|$.
\end{example}
\section{Gr\"obner-Shirshov bases theory for Lie superalgebras} \label{S:GS}
In this section we recall the theory of Gr\"obner-Shirshov bases and Composition-Diamond lemma for the case of Lie superalgebras (see \cite{B2}). In fact, Composition-Diamond lemma is a powerful theorem in algorithmic
and combinatorial algebra which provides linear bases and normal forms
of the elements of an algebra presented by generators and defining relations.

 Let us consider a free Lie superalgebra and an graded ideal $\Id(S)$ generated by a set $S$. If $S$ is a Gr\"obner-Shirshov basis, then the leading term $\bar{f}$ of any polynomial $f$ in $\Id(S)$ contains $\bar{s}$ as a subword, for some $s \in S$ (see below Lemma~\ref{L:GS}). In the sequel, we re-express this idea with more details. For this purpose, we need an ordered set of generators as well as a monomial ordering. It is worth pointing out that the property of being a Gr\"obner-Shirshov basis is always relative to a specific monomial ordering.

Let $T=T_{\bar{0}} \cup T_{\bar{1}}$ be a $\mathbb{Z}_2$-graded set with a linear ordering $\prec$, and let $T^{\ast}$ (resp., $T^{\#}$) be the semigroup of associative words on $T$ (resp., the groupoid of nonassociative words on $T$). There exist induced $\mathbb{Z}_2$-gradings both on the semigroup $T^{\ast}$ and on the groupoid $T^{\#}$: $T^{\ast} = T_{\bar{0}}^{\ast} \cup T_{\bar{1}}^{\ast}$ and $T^{\#}=T_{\bar{0}}^{\#} \cup T_{\bar{1}}^{\#}$, respectively.
Indices $\bar{0}$ and $\bar{1}$ indicate the \emph{even}  and  \emph{odd} elements, respectively. The \emph{length} of a word $u$ is denoted by $l(u)$ and the empty word is denoted by $1$.

\begin{remark}
For an associative word $u \in T^{\ast}$ there exist a certain arrangement of brackets denoted by $(u)$ and a canonical bracket-removing homomorphism $\rho \colon T^{\#} \to T^{\ast}$ given by $\rho((u))=u$ for $u \in T^{\ast}$.
\end{remark}
Let us define the \emph{lexicographical} and \emph{length-lexicographical} orderings, which will be denoted by $<$ and $\ll$ on $T^{\ast}$, respectively, as follows:
\begin{itemize}
    \item [(i)] $u < 1$ for any nonempty word $u$; and inductively, $u < v$ whenever $u=x_iu^{\prime}$, $v=x_jv^{\prime}$, and $x_i \prec x_j$ or $x_i=x_j$ and $u^{\prime} < v^{\prime}$.
    \item[(ii)] $u \ll v$ if $l(u) < l(v)$ or $l(u)=l(v)$ and $u < v$.
\end{itemize}
\begin{remark}
The orderings $<$ and $\ll$ are defined on $T^{\#}$ by
\begin{itemize}
    \item $u < v$ if and only if $\rho(u) < \rho(v)$.
    \item $u \ll v$ if and only if $\rho(u) \ll \rho(v)$.
\end{itemize}
\end{remark}

 \begin{definition}
A nonempty word $u$ is called a \emph{Lyndon-Shirshov word} if $u \in T$ or for $vw > wv$ for any decomposition of $u=vw$ with $v,w \in T^{\ast}$, i.e. $u$ is greater than any of its cyclic permutations. A nonempty word $u$ is called a \emph{super-Lyndon-Shirshov word} if either it is a Lyndon-Shirshov word or it has the form $u=vv$ with $v$ a Lyndon-Shirshov word in $T_{\bar{1}}^{\ast}$.
\end{definition}

\begin{definition}
A nonempty nonassociative word $u$ is called a \emph{Lyndon-Shirshov monomial} if  either $u$ is an element of $T$ or
\begin{itemize}
    \item [(i)] if $u=u_1u_2$, then $u_1 , u_2$ are Lyndon-Shirshov monomials with $u_1> u_2$,
    \item[(ii)] if $u=(v_1v_2)w$, then $v_2 \leq w$.
\end{itemize}
A nonempty nonassociative word $u$ is called a \emph{super-Lyndon-Shirshov monomial } if either it is a Lyndon-Shirshov monomial or it has the form $u=vv$ with $v$ a Lyndon-Shirshov monomial in $T_{\bar{1}}^{\#}$.
\end{definition}
If we denote by $A\langle T \rangle$ the free associative algebra on $T$ over a field $\mathbb{K}$ with $\chr(\mathbb{K})\neq 2,3$, then we can define a the Lie superalgebra structure on $A\langle T \rangle$ with the superbracket
\[ [x,y]= xy - (-1)^{|x| |y|} yx ,\]
for any $x,y \in A\langle T \rangle$.
Denote by $L \langle T \rangle $  the free Lie superalgebra generated by $T$, which is a graded subalgebra of $A\langle T \rangle$ with the above superbracket. As a direct consequence, expanding superbracket gives us associative words belonging to $ A\langle T \rangle$.

There is a one to one correspondence between the set of super-Lyndon-Shirshov words and the set of super-Lyndon-Shirshov monomials. In the sequel, based on Shirshov's special bracketing \cite{S1}, the bracketing of super-Lyndon-Shirshov words is described. It has a key role in the definition of compositions between polynomials in free Lie superalgebras.

\begin{lemma}[\cite{B2}]
Let $u$ and $v$ be super-Lyndon-Shirshov words such that $v$ is contained in $u$ as a subword. Write $u=avb$ with $a,b\in T^{\ast}$. Then there is an arrangement of brackets $[u]=(a[v]b)$
on $u$ such that $[v]$ is a super-Lyndon-Shirshov monomial, $\overline{[u]}=u$ and the leading coefficient of $[u]$ is either $1$ or $2$.
\end{lemma}
Given a super-Lyndon-Shirshov word $u=avb$, where $v$ is another super-Lyndon-Shirshov word and $a,b\in T^{\ast}$, the \emph{bracketing relative to}  $v$, $[u]_v$,  is defined as follows:
\begin{itemize}
    \item [(i)] $[u]_v=(a[v]b)$ if the leading coefficient of $[u]$ is $1$,
    \item[(ii)] $[u]_v=\dfrac{1}{2}(a[v]b)$ if the leading coefficient of $[u]$ is $2$.
\end{itemize}
We note that $[u]_v$ is monic and $\overline{[u]_v}=u$. Let also $p$ be a monic polynomial in $L\langle T\rangle$ such that $\bar{p}$ is a super-Lyndon-Shirshov word. The \emph{bracketing on u relative to p} $[u]_p$ is defined to be the result of the substitution of $p$ instead of $\bar{p}$ in $[u]_{\bar{p}}$. Again, $[u]_p$ is monic and $\overline{[u]_p}=u$.

\subsection*{Composition of polynomials}
\begin{enumerate}
    \item \textbf{Associative compositions}. Let $f,g$ be monic elements in the free associative algebra $A\langle T \rangle$ with leading terms $\bar{f}$ and $\bar{g}$.

   If we have $w=\bar{f}a=b\bar{g}$ with $l(\bar{f}) > l(b)$ and $a,b \in T^{\ast}$ then the \emph{intersection composition} is
    \[ (f,g)_w=fa - bg.\]
     If $w=\bar{f}=a \bar{g} b $, then the \emph{inclusion composition} is defined as
    \[(f,g)_w=f-agb .\] Note that $\overline{(f,g)_w} \ll w$.
    \item \textbf{Lie compositions}. Let $f,g$ be monic polynomials in the free Lie superalgebra $L\langle T \rangle$ with leading term $\bar{f}$ and $\bar{g}$.

    If there exist $a,b \in T^{\ast}$ such that $\bar{f}a=b\bar{g}=w$ with $l(\bar{f}) > l(b)$, then the \emph{intersection composition} is defined as
    \[ \langle f,g \rangle_{w} = [w]_f - [w]_g .\]

    If there exist $a,b \in T^{\ast}$ such that $\bar{f}=a\bar{g}b=w$, then the \emph{inclusion composition} is defined as
    \[ \langle f,g \rangle_w=f-[w]_g .\] We note that $ \overline{\langle f,g \rangle} \ll w$.
\end{enumerate}

Associative and Lie compositions are closely related, as it is shown in~\cite[Lemma 2.7 and Theorem 2.8]{B2}. In fact, there is an equivalent relationship between these two types of compositions. Henceforth, we will refer just to Lie compositions.

In order to define Gr\"obner-Shirshov bases, it is necessary to provide the triviality criteria of compositions. Sometimes, in the literature, the concept of triviality of compositions has been referred to as \emph{closed under compositions} or \emph{complete under compositions}.
\subsection*{Triviality criteria}
Let $S$ be a set of monic polynomials in $L\langle T \rangle \subset A\langle T \rangle$ and $\Id(S)$ be the graded ideal generated by $S$ in the free Lie superalgebra $L\langle T \rangle$.

If for any $f,g\in S$ such that the Lie composition $\langle f,g\rangle_w$ is defined with respect to a certain $w\in T^*$, we have \[\langle f,g\rangle_w= \sum \alpha_i a_i s_i b_i ,\]
where $\alpha_i \in \mathbb{K}$, $a_i , b_i \in T^{\ast}$ and $s_i \in S$, with $a_is_ib_i \ll w$, we say that $S$ is \emph{trivial under Lie composition}.

\begin{definition}
A set of monic polynomials $S$ is called a Gr\"obner-Shirshov basis for the graded ideal $\Id(S)$, if it is trivial under Lie composition.
\end{definition}

In~\cite{B2} it is explained how to construct a Gr\"obner-Shirshov basis for any graded ideal $\Id(S)$ of $L\langle T\rangle$.

The next lemma is based on Shirshov's lemma. According to this lemma, the leading term of every monic polynomial in the graded ideal generated by $S$ can be reduced with respect to an element in $S$. The main consequence of the lemma is that we are able to determine a linear basis for the factor algebra $L\langle T \rangle /\Id(S)$ by taking the set of irreducible monomials.
\begin{lemma}[\cite{S1}] \label{L:GS}
If $S$ is a Gr\"obner-Shirshov basis for the graded ideal $\Id(S)$, then for any $f \in \Id(S)$, we can express the word $\bar{f}$ as $\bar{f}=a\bar{s}b$, where $s \in S$ and $a,b \in T^{\ast}$.
\end{lemma}

\section{Construction of the HNN-extension} \label{S:HNN}
In this section we construct the  HNN-extension of Lie superalgebras. Let $L$ be a Lie superalgebra over a field $\mathbb{K}$ with $\chr(\mathbb{K})\neq 2,3$, and $A$ be a graded subalgebra. Assume that $d\colon A\to L$ is a derivation defined on the graded subalgebra $A$. We can restrict to consider homogeneous derivations. We define the HNN-extension as the Lie superalgebra
\[
H \coloneqq \langle L, t: d(a)=[t,a], \ a\in A \rangle.
\]
Here $t$ is a new symbol of degree $|t|=|d|$ not belonging to $L$ which is added to the presentation of $L$. We also add the relation $[t,a]=d(a)$ (where $a\in A$). In the following, we give an equivalent presentation with respect to structural constants and check the triviality criteria for compositions of polynomials in the newly defined presentation.

Let us consider a linear basis $X$ of $L$ including a basis of $A$, which will be denoted by $B$, and a total ordering $ B < X\setminus B< t$.
We also consider the following superbrackets:
\[[x,y]=\sum_{v\in X} \alpha_{xy}^{v} v\]
%\begin{itemize}
%    \item $[x,y]=\sum_{v\in X} \alpha_{xy}^{v} v$
%    \item $[x,x]=\sum_v \alpha_{xx}^{v} v$
%\end{itemize}
for $x,y \in X$. Since $A$ is a subalgebra, then $\alpha_{ab}^v=0$ for $a,b \in B$ and $v \notin B$. In addition, we have
\[ d(a)= \sum_{v\in X} \beta_{a}^{v} v \]
for $a \in B$. Therefore, we have the following presentation for the HNN-extension of the Lie superalgebra $L$
\begin{align*}
    H \coloneqq \langle X,t \mid &[x,y]=\sum_{v\in X} \alpha_{xy}^{v} v, ~ [x_{\bar{1}},x_{\bar{1}}]=\sum_{v\in X} \alpha_{x_{\bar{1}}x_{\bar{1}}}^v v,\\
    & [t,a]=\sum_{v\in X} \beta_{a}^{v} v,\quad  x,x_{\bar{1}},y \in X, ~ x>y ,~ |x_{\bar{1}}|=\bar{1}, ~ a \in B \rangle.
\end{align*}

Now, we research the possible compositions of the presentation, and check if they are trivial. First, recall from Section~\ref{S:prel} that, due to the Jacobi identity, the structural constants satisfy
\begin{equation}\label{jac}
\sum_{v\in X}(\alpha_{yz}^v\alpha_{xv}^{u}-\alpha_{xy}^v\alpha_{vz}^{u}-(-1)^{|x||y|}\alpha_{xz}^v\alpha_{yv}^{u})=0,
\end{equation}
for $x,y,z,v,u\in X$, and where $\alpha_{xy}^v=-(-1)^{|x||y|}\alpha_{yx}^{v}$ due to the antisymmetry. In particular, note that $\alpha_{xx}^v=0$ if $|x|=\overline{0}$.
Also, the condition of derivation $d([x,y])=[d(x),y]+(-1)^{|x||y|}[x,d(y)]$ is written in terms of structural constants as
\begin{equation}\label{der}
\sum_{c\in B}\alpha_{ab}^c\beta_c^{u}=\sum_{v\in X}(\beta_a^v\alpha_{vb}^{u}+(-1)^{|d||a|}\beta_b^v\alpha_{av}^{u}),
\end{equation}
for $a,b,c\in B$ and $v,u\in X$.

We denote
\begin{align*}
f_{xy}&=[x,y]-\sum_{v\in X}\alpha_{xy}^vv,\\
f_{xx}&=[x,x]-\sum_{v\in X}\alpha_{xx}^vv,
\end{align*}
 and
\[g_a=[t,a]-\sum_{v\in X}\beta_a^vv.\]
Note that considering $f_{xx}$ is only meaningful when $|x|=\overline{1}$. We also consider $S$ as the following set of monic polynomials in $L\langle X\cup\{t\}\rangle$:
\[S =\{ f_{xy}, f_{xx}, g_{a} \mid x,y \in X , a \in B \}. \]

Now, let us check all the possible compositions of Lie polynomials $f_{xy},f_{xx},g_a$. Setting $x > y > z$ and $a> b$, it is clear that the leading terms in $A\langle X\cup\{t\} \rangle$ of these polynomials are $\overline{f_{xy}}=xy$, $\overline{f_{xx}}=2xx$, and $\overline{g_a}=ta$. So, the unique Lie compositions are the following: $\langle f_{xy},f_{yz}\rangle_{xyz}$, $\langle f_{xy},f_{yy}\rangle_{xyy}$, $\langle f_{xx},f_{xy}\rangle_{xxy}$, $\langle g_a,f_{ab}\rangle_{tab}$, $\langle g_a,f_{aa}\rangle_{taa}$. Note that all of them are compositions of intersection.
Just two of these five compositions are trivial, as it is seen next.

\begin{align*}
\langle f_{xy},f_{yz}\rangle_{xyz}&=[f_{xy},z]-[x,f_{yz}] \\
& =[[x,y],z]-\sum_{v\in X}\alpha_{xy}^v[v,z]-[x,[y,z]]+\sum_{v\in X}\alpha_{yz}^v[x,v]\\
& =-(-1)^{|x||y|}[y,[x,z]]-\sum_{v\in X}\alpha_{xy}^v[v,z]+\sum_{v\in X}\alpha_{yz}^v[x,v]\\
&=-(-1)^{|x||y|}[y,[x,z]-\sum_{v\in X}\alpha_{xz}^vv]-(-1)^{|x||y|}\sum_{v\in X}\alpha_{xz}^v([y,v]-\sum_{u\in X}\alpha_{yv}^{u}u)\\
& \ \ -\sum_{v\in X}\alpha_{xy}^v([v,z]-\sum_{u\in X}\alpha_{vz}^uu)+\sum_{v\in X}\alpha_{yz}^v([x,v]-\sum_{u\in X}\alpha_{xv}^uu)\\
& \ \ -\sum_{u\in X}\Big(-(-1)^{|x||y|}\sum_{v\in X}\alpha_{xz}^v\alpha_{yv}^{u}+\sum_{v\in X}\alpha_{xy}^v\alpha_{vz}^u+\sum_{v\in X}\alpha_{yz}^v\alpha_{xv}^u\Big)u\\
&=-(-1)^{|x||y|}[y,f_{xz}]-(-1)^{|x||y|}\sum_{v\in X}\alpha_{xz}^vf_{yv}-\sum_{v\in X}\alpha_{xy}^vf_{vz}+\sum_{v\in X}\alpha_{yz}^vf_{xv}.
\end{align*}

We see that $\overline{\langle f_{xy},f_{yz}\rangle_{xyz}}=xzy<xyz$, so the composition is trivial.

\begin{align*}
\langle g_a,f_{ab}\rangle_{tab}&=[g_a,b]-[t,f_{ab}] \\
& =[[t,a],b]-\sum_{v\in X}\beta_a^v[v,b]-[t,[a,b]]+\sum_{v\in B}\alpha_{ab}^v[t,v] \\
& =-(-1)^{|t||a|}[a,[t,b]]-\sum_{v\in X}\beta_a^v[v,b]+\sum_{v\in B}\alpha_{ab}^v[t,v] \\
& =-(-1)^{|t||a|}[a,[t,b]-\sum_{v\in X}\beta_b^vv] -(-1)^{|t||a|} \sum_{v\in X}\beta_b^v([a,v]-\sum_{u\in X}\alpha_{av}^{u}u)  \\
& \ \ -\sum_{v\in X}\beta_a^v([v,b]-\sum_{u\in X}\alpha_{vb}^{u}u)+\sum_{v\in B}\alpha_{ab}^v([t,v]-\sum_{u\in X}\beta_v^{u}u) \\
&   \ \  +\sum_{u\in X}\Big(-(-1)^{|t||a|}\sum_{v\in X}\beta_b^v\alpha_{av}^{u}-\sum_{v\in X}\beta_a^v\alpha_{vb}^{u}+\sum_{v\in B}\alpha_{ab}^v\beta_v^{u}\Big)u \\
& = -(-1)^{|t||a|}[a,g_b] -(-1)^{|t||a|} \sum_{v\in X}\beta_b^vf_{av}-\sum_{v\in X}\beta_a^vf_{vb}+\sum_{v\in B}\alpha_{ab}^vg_v.
\end{align*}
Again, we have that  $\overline{\langle g_a,f_{ab}\rangle_{tab}}=tba<tab$, so this composition is also trivial.

However, it holds that, if $|x|=\overline{1}$,
\begin{align*}
\langle f_{xx},f_{xy}\rangle_{xxy}&=\frac{1}{2}[f_{xx},y]-[x,f_{xy}] \\
& =\frac{1}{2}[[x,x],y]-\frac{1}{2}\sum_{v\in X}\alpha_{xx}^v[v,y]-[x,[x,y]]+\sum_{v\in V}\alpha_{xy}^v[x,v] \\
& =-\frac{1}{2}[[x,x],y]+[x,[x,y]]-\frac{1}{2}\sum_{v\in X}\alpha_{xx}^v[v,y]+\sum_{v\in V}\alpha_{xy}^v[x,v] \\
& = [-\frac{1}{2}[x,x]+\frac{1}{2}\sum_{v\in X}\alpha_{xx}^vv,y]-\frac{1}{2}\sum_{v\in X}\alpha_{xx}^v([v,y]-\sum_{u\in X}\alpha_{vy}^{u}u)\\
&\quad +[x,[x,y]-\sum_{v\in X}\alpha_{xy}^vv]+\sum_{v\in X}\alpha_{xy}^v([x,v]-\sum_{u\in X}\alpha_{xv}^uu)\\
&\quad -\frac{1}{2}\sum_{v\in X}\alpha_{xx}^v([v,y]-\sum_{u\in X}\alpha_{vy}^{u}u)+\sum_{v\in V}\alpha_{xy}^v([x,v] -\sum_{u\in X}\alpha_{xv}^{u}u)\\
&\quad +\sum_{u\in X}\Big(-\sum_{v\in X}\alpha_{xx}^v\alpha_{vy}^{u} +\sum_{v\in X}\alpha_{xy}^v\alpha_{xv}^u  +\sum_{v\in V}\alpha_{xy}^v\alpha_{xv}^{u}\Big)u \\
& =-\frac{1}{2}[f_{xx},y]+[x,f_{xy}]-\sum_{v\in X}\alpha_{xx}^vf_{vy}+2\sum_{v\in X}\alpha_{xy}^vf_{xv};
\end{align*}
the leading monomial is $\overline{\langle f_{xx},f_{xy}\rangle_{xxy}}=xxy$, so the composition is not trivial. Similarly, if $|y|=\overline{1}$, we have
\begin{equation*}
\langle f_{xy},f_{yy}\rangle_{xyy}=[x,\frac{1}{2}f_{yy}]-(-1)^{|x||y|}[y,f_{xy}]-2\sum_{v\in X}\alpha_{xy}^vf_{vy}+\sum_{v\in X}\alpha_{yy}^vf_{xv}.
\end{equation*}
The leading monomial is $\overline{\langle f_{xy},f_{yy}\rangle_{xyy}}=xyy$, so the composition is, again, not trivial.

Finally, we research the last composition $\langle g_a,f_{aa}\rangle_{taa}$, for $|a|=1$.

\begin{align*}
\langle g_a,f_{aa}\rangle_{taa}&=[g_a,a]-\frac{1}{2}[t,f_{aa}] \\
&=[[t,a],a]-\sum_{v\in X}\beta_a^v[v,a]-\frac{1}{2}[t,[a,a]]+\frac{1}{2}\sum_{v\in B}\alpha_{aa}^v[t,v] \\
&=\frac{1}{2}[t,[a,a]]-(-1)^{|t||a|}[a,[t,a]]-\sum_{v\in X}\beta_a^v[v,a]+\frac{1}{2}\sum_{v\in B}\alpha_{aa}^v[t,v] \\
&=\frac{1}{2}[t,[a,a]-\sum_{v\in X}\alpha_{aa}^vv]+\frac{1}{2}\sum_{v\in B}\alpha_{aa}^v([t,v]-\sum_{u\in X}\beta_{v}^{u}u) \\
&\quad -(-1)^{|t||a|}[a,[t,a]-\sum_{v\in X}\beta_{a}^vv]-(-1)^{|t||a|}\sum_{v\in X}\beta_{a}^v([a,v]-\sum_{u\in X}\alpha_{av}^{u}u)  \\
&\quad -\sum_{v\in X}\beta_a^v([v,a]-\sum_{u\in X}\alpha_{va}^uu)+\frac{1}{2}\sum_{v\in B}\alpha_{aa}^v([t,v]-\beta_{v}^uu) \\
&\quad +\sum_{u\in X}\Big(\sum_{v\in B}\alpha_{aa}^v\beta_{v}^{u}-(-1)^{|t||a|}\sum_{v\in X}\beta_{a}^v\alpha_{av}^{u}-\sum_{v\in X}\beta_a^v\alpha_{va}^u\Big)u \\
&=\frac{1}{2}[t,f_{aa}]-(-1)^{|t||a|}[a,g_a]-(-1)^{|t||a|}\sum_{v\in X}\beta_a^vf_{av}-\sum_{v\in X}\beta_a^vf_{va}+\sum_{v\in X}\alpha_{aa}^vg_{v}.
\end{align*}

As the leading monomial is $\overline{\langle g_a,f_{aa}\rangle_{taa}}=taa$, the composition is not trivial.

The calculation of the possible compositions of polynomials in a presentation of a HNN-extension of Lie superalgebras helps us to extract the general idea about the degree of polynomials of its Gr\"obner-Shirshov basis. In fact, it is an experimental technique based on Shirshov's algorithm \cite{S1} which shows that by extending $S$ to a larger set, the new polynomials are of degree at least two. This proves that the Lie superalgebra $L$ embeds in $H$. So, we have the following theorem as a direct result of our computations.
\begin{theorem}
Every Lie superalgebra embeds into its HNN-extensions.
\end{theorem}

\section{Application} \label{S:app}

The work \cite{M2} of Mikhalev provided an embeddability theorem for Lie superalgebras analogue to a known theorem in group theory (see \cite{H1}).
Mikhalev proved that every Lie superalgebra of at most countable dimension can be embedded in a Lie superalgebra with two generators. In the following, we show that Mikhalev's approach is compatible with the concept of HNN-extension of Lie superalgebras.

For the next lemma, we use semidirect product of Lie superalgebras (see \cite{M1}).

\begin{lemma}
Let $L$ be a Lie superalgebra and $A$ a graded subalgebra of $L$ which is free on a subset $Z$. Then every map $d \colon Z \to L$ can be extended to a derivation from $A$ to $L$.
\end{lemma}

\begin{proof}
Let $L^{\prime}$ be an abelian Lie superalgebra with underlying linear space $L$. Since $L^{\prime}$ is abelian, then $L_{\bar{0}}^{\prime}=0$, i.e. $[x,y]=0$ for all $x,y \in L^{\prime}$.
Let us define the morphism $h \colon A \to \Der(L^{\prime})$ by $h(a)(l)=[l,a]$.
Let us consider $A \ltimes L^{\prime}$ ($A$ acts  on $L^{\prime}$ via the morphism $h$) and a Lie superalgebra homomorphism $f \colon A \to A \ltimes L^{\prime}$ defined by $f(z)=(z,d(z))$ for every $z \in Z$.
For an arbitrary
$a \in A$, let $f(a) = (u(a),D(a))$.
Then for  $a_1, a_2\in A$ we have
\begin{align*}
(u([a_1, a_2]), D([a_1, a_2]))&= f([a_1, a_2])\\
                              &=[f(a_1), f(a_2)]\\
                              &=[(u(a_1), D(a_1)), (u(a_2), D(a_2))]\\
                              &=([u(a_1), u(a_2)], [D(a_1), D(a_2)]-(-1)^{{|D(a_1)|}  |u(a_2)|}[u(a_2), D(a_1)]\\
                              & \qquad +[u(a_1), D(a_2)]).
\end{align*}

Since $L^{\prime}$ is abelian, $[D(a_1), D(a_2)]=0$. Also, the above computation shows that $u(a)=a$ for all $a$. Hence if we use $[a_1,a_2]=-(-1)^{|a_1||a_2|}[a_2,a_1]$, we have
\[
D([a_1, a_2])= [D(a_1), a_2]+(-1)^{{|D(a_2)| |a_1|}}[a_1, D(a_2)],
\]
so $D$ is a derivation and extends $d$.

\end{proof}
\begin{lemma}[\cite{M2}]\label{essentiallemma}  The elements $[y, \dots [y, [y,x]],\dots ]$ form a free Lie superalgebra $L(x,y)$.
\end{lemma}
\begin{theorem}[\cite{M2}]
Let $\chr (\mathbb{K})\neq 2,3$. Every Lie superalgebra $L$ over $\mathbb{K}$ with at most countable dimension can be embedded in a Lie superalgebra with two generators.
\end{theorem}
\begin{proof}
Let us consider the free product $L_1=L \ast L(a,b)$, where $L(a,b)$ is the free Lie superalgebra generated by $a,b$. The elements \[[b,\dots,[b,[b,a]],\dots ]\] are a free basis $Z$ of some graded subalgebra $A$ of $L_1$.
We consider the map $d \colon Z \to L_1$ with $d(a)=b$ and $d([b^{n},a])=c_{n}$ (by $[b^{n},a]$
 we mean $[b,\dots , [b,a] \dots ]$ where $b$ appears $n$ times) where the $c_{n}$ conform a generating set of $L$. Then, this map $d$ extends to a derivation.
Now the HNN-extension is defined  as follows
\[ \langle L_1,t \mid [t,[b^{n},a]]=c_{n} , [t,a]=b \rangle .\]
It is a Lie superalgebra generated by $a$ and $t$ and contains $L$.
\end{proof}

\section*{Acknowledgments}

The authors were supported by  Agencia Estatal de Investigaci\'on (Spain), grant MTM2016-79661-P (European FEDER support included, UE).
 The first author was also partially supported by  visiting
program (USC).
P. P\'aez-Guill\'an was also supported by FPU scholarship, Ministerio de Educaci\'on, Cultura
y Deporte (Spain).

\end{document}